\documentclass{article}
\usepackage[utf8]{inputenc}
\usepackage{amsthm,amssymb,geometry}
\usepackage{tikz}
\usepackage{mathdots}
\usepackage{indentfirst}
\usepackage{hyperref}
\usepackage{thmtools}
\usepackage[capitalize]{cleveref}

\newcommand{\N}{\mathbb{N}}
\newcommand{\twochildren}[4]{child{node(#1){\null}} child{node(#2){\null}} child{node(#3){\null}} child{node(#4){\null}}}

\newtheorem{theorem}{Theorem}[section]
\newtheorem{definition}[theorem]{Definition}
\newtheorem{lemma}[theorem]{Lemma}
\newtheorem{cor}[theorem]{Corollary}
\newtheorem{conj}{Conjecture}

\title{Distance Sequences of Locally Infinite Primitive Graphs}
\author{Katalin Berlow}
\date{\today}

\begin{document}

\maketitle

\begin{abstract}
        A graph is called \textit{primitive} if its automorphism group acts primitively on the vertex set. In this paper, we prove a classification of the possible distance sequences of locally infinite primitive graphs. In particular we show that if a primitive graph is locally uncountable, the distance sequence is constant until it terminates. We also prove a constraint on the distance sequences of locally finite infinite graphs. 
\end{abstract}

\section{Introduction}\label{sec:Intro}

    W. Pegden in \cite{Pegden} has noted that for locally infinite vertex-transitive graphs, the possible distance sequences are limited. Pegden proved that the distance sequence of a locally infinite vertex-transitive graph must be constant (if infinite) or must be constant until the last entry, where it may decrease (if finite). Constructions were also given, thus proving each sequence occurs. However, these constructions were for the most part imprimitive and thus raising the question: Which sequences occur in the primitive case? In this paper we prove some constraints on the distance sequences of locally infinite primitive vertex-transitive graphs and give constructions for all but one case. 

    Throughout this paper $\Gamma$ will be used to denote a graph, $V(\Gamma)$ its vertex set and $E$ its edge relation. We will begin with some definitions. 

     \begin{definition}
    A graph $\Gamma$ is vertex-transitive if all vertices are equivalent under its automorphism group. 
    \end{definition}
    
    One can think about vertex-transitive graphs as those graphs which look the same from the perspective of any vertex. 
    
     \begin{definition}
    The \textit{distance sequence} of a vertex-transitive graph $\Gamma$ is defined as follows: Given any vertex $v\in V(\Gamma)$, let $d(i)$ be the number of vertices at distance $i$ from $v$. 
    \end{definition}
    
    Note that for a vertex-transitive graph, this sequence $d$ is not dependent on the vertex $v$ to which it refers. 
    
    \begin{definition}
    A graph $\Gamma$ is distance-transitive if for any $n\in \N$, any pair of vertices at distance $n$ is equivalent under the automorphism group.  
    \end{definition}
    
 Let $G$ be a permutation group acting on a set $X$. 

    \begin{definition}
    A $G-$congruence is a $G-$invariant equivalence relation such that $x \equiv y \Leftrightarrow g\cdot x \equiv g \cdot y$ for all $g \in G$. Equivalence classes of the $G$-congruence are called blocks. The sets $\varnothing, \{x\},$ and $X$ are always blocks and are thus considered \textit{trivial blocks}.  
    A set of such blocks is called a \textit{block system}. If all blocks in a block system are trivial, then the block system is trivial.   
    \end{definition}
    
    \begin{definition}
    If there exists a nontrivial block system, then $G$ acts imprimitively on $X$. Otherwise $G$ acts primitively. 
    \end{definition}
    
    \begin{definition} A graph $\Gamma$ is primitive if the automorphism group of $\Gamma$ acts primitively on the vertex set. 
    \end{definition}
    
        Note that all primitive graphs are vertex-transitive. To see this assume $\Gamma$ is a graph which is not vertex-transitive. Let $v$ and $u$ be vertices such that no automorphism sends $v$ to $u$. Let $A$ be the set of vertices which can be sent to $v$ through an automorphism. Then $A$ and $V(\Gamma)\setminus A$ are blocks of imprimitivity. 
        
        Also note that all primitive graphs are connected. This is because if a graph is not connected, then the connected components form blocks of imprimitivity. 

    If $\Gamma$ is a graph with diameter $m$, then for $i<m$ and $v\in V(\Gamma)$ we let $\Gamma_i(v)$ denote the set of vertices at distance $i$ from $v$.

    \begin{definition}
    A graph $\Gamma$ of diameter $m$ is said to be \textit{antipodal} if for any vertex $u$ and for all distinct vertices $v,w\in \Gamma_0(u)\cup \Gamma_m(u)$, we have $d(v,w) = m$.
    \end{definition}
    
    Consider the following theorem of Smith \cite{Smith}.
    
    \begin{theorem}[Smith]\label{thm: smith}
    Let $\Gamma$ be a distance-transitive graph with valency $k>2$ and diameter $>2$. Then $\Gamma$ is imprimitive iff it is bipartite or antipodal. 
    \end{theorem}
    
    Though Smith stated \cref{thm: smith} for finite graphs, his proof in \cite{Smith} works
for infinite graphs as well.

 In \cref{sec: uncountable} we will prove the following theorem.

\begin{theorem}\label{thm uncountable}
    If $\Gamma$ is a locally infinite primitive graph of cardinality $\kappa>\aleph_0$, then it has distance sequence $1,\kappa,\kappa,\dots$ if its diameter is infinite and it has distance sequence $1,\kappa,\dots,\kappa$ where $\kappa$ appears $n$ times if its diameter is $n$. 
    \end{theorem}
    
 We will also give constructions to show that all such sequences occur. In \cref{sec: countable}, we prove the corresponding result for the countable case. This case is more nuanced.

\begin{theorem}\label{thm: countable}
    Let $\Gamma$ be a countable, locally infinite, primitive graph. If $\Gamma$ has infinite diameter, its distance sequence is $1,\aleph_0,\aleph_0,\dots$. If $\Gamma$ has finite diameter then $\Gamma$ has distance sequence $1,\aleph_0,\dots,\aleph_0,a$ for $2\leq a \leq \aleph_0$.  
    \end{theorem}

    We also provide constructions for all cases except for when the distance sequence ends in a prime or 4. This case remains open.

\section{Uncountable Graphs}\label{sec: uncountable}

\begin{theorem}\label{uncountable2}
    If $\Gamma$ is a locally infinite primitive vertex-transitive graph of cardinality $\kappa>\aleph_0$, then it has distance sequence $1,\kappa,\kappa,\dots$ if its diameter is infinite and it has distance sequence $1,\kappa,\dots,\kappa$ where $\kappa$ appears $n$ times if its diameter is $n$. 
    \end{theorem}

It was shown by Pegden in \cite{Pegden}, that if $\Gamma$ is a locally infinite vertex-transitive graph of degree $\kappa$ with infinite diameter then it has distance sequence $1,\kappa,\kappa,\dots$ and if it is of finite degree, then it has distance sequence $1,\kappa,\dots,\kappa,\alpha$ for $\alpha\leq \kappa$. Thus it suffices to show that in the second case, $\alpha$ must equal $\kappa$.

\begin{proof}

Let $\Gamma$ be such a graph with diameter $m$ and distance sequence \\ $1,\kappa,\dots,\kappa,\alpha$ for $\kappa>\alpha$. 
Define a relation $\sim$ on $V(\Gamma)$ by $$v\sim w\iff d(v,w) = m$$. The relation $\sim$ is symmetric, and thus $\Gamma':= (V(\Gamma),\sim)$ is a graph.  

Note that $\Gamma'$ is also vertex-transitive and has degree $\alpha$. Since $|V(\Gamma)| = \kappa > \aleph_0$ and $\Gamma'$ has degree $\alpha<\kappa$, $\Gamma'$ is not connected and has $\kappa$ many connected components.  Let $B:= \{B_\alpha: \alpha< d\}$ be the partition of $V(\Gamma)$ into connected components. 

Let $\varphi$ be an automorphism of $\Gamma$ and $a,b\in B_\alpha$ for some $\alpha < \kappa$. Then there is a path connecting $a$ to $b$ in $\Gamma'$, let $a= a_0, a_1, \dots , a_{k-1}=b$ be this path where $k$ is the distance between $a$ and $b$ in the graph $\Gamma'$. Then for all $i<k-1$, $d(a_i,a_{i+1})=m$. Thus for all $i<k-1$, $d(\varphi(a_i),\varphi(a_{i+1}))=m$ and therefore $\varphi(a_0),\dots,\varphi(a_{k-1})$ is a path in $\Gamma'$. Thus $\varphi(a)$ is connected to $\varphi(b)$ in $\Gamma'$ and therefore are in the same connected component. Thus $B$ is a block system of $\Gamma$ implying $\Gamma$ is imprimitive. 
\end{proof}

Next we will provide constructions to show that for any infinite cardinal $\kappa$, there are graphs with distance sequence $1,\kappa,\kappa,\dots$ and $1,\kappa,\dots,\kappa$. 

\begin{theorem} \label{thm: triangle tree}
There is a primitive graph with distance sequence $1,\kappa,\kappa,\dots$ for any infinite cardinal $\kappa$. 
\end{theorem}

\begin{center}
\begin{figure}
\tikzstyle{every node}=[circle, draw, fill=black!50,
                        inner sep=0pt, minimum width=4pt]
\begin{tikzpicture}[style=thick,scale=0.65] 
\tikzstyle{every node}=[circle,fill=black!50,inner sep=1pt,draw] 
\tikzstyle{level 1}=[sibling distance=40mm]
\tikzstyle{level 2}=[sibling distance=20mm] 
\tikzstyle{level 3}=[sibling distance=10mm] 
\node(a){\null}[grow'=down] 
child{node(e1){\null} }
child{node(f1){\null} }
child{node(e2){\null} }
child{node(f2){\null} 
\twochildren{c3}{d3}{c4}{d4} }; 
\draw{(e1.west) -- (f1.east)};
\draw{(e2.west) -- (f2.east)};
\draw{(c3.west) -- (d3.east)};
\draw{(c4.west) -- (d4.east)};
\draw{(a.east) -- (10,-1.5)};
\draw{(a.east) -- (12,-1.5)};
\draw{(f2.east) -- (-2,-3)};
\draw{(f2.east) -- (-1,-3)};
\draw{(-1,-3) -- (-2,-3)};
\draw{(10,-1.5) -- (12,-1.5)};
\filldraw[fill=black!50](10,-1.5) circle (.1);
\filldraw[fill=black!50](12,-1.5) circle (.1);
\filldraw[fill=black!50](-2,-3) circle (.1);
\filldraw[fill=black!50](-1,-3) circle (.1);
 \node[fill = white, draw = white] at (13,-1.5) {\dots};
 \node[fill = white, draw = white] at (11,-2.5) {$\ddots$};
 \node[fill = white, draw = white] at (-6,-4) {$\vdots$};
 \node[fill = white, draw = white] at (0,-2.5) {$\dots$};
\end{tikzpicture} 
    \caption{ The Graph $\Gamma^\kappa$}\label{Fig1}
\end{figure}
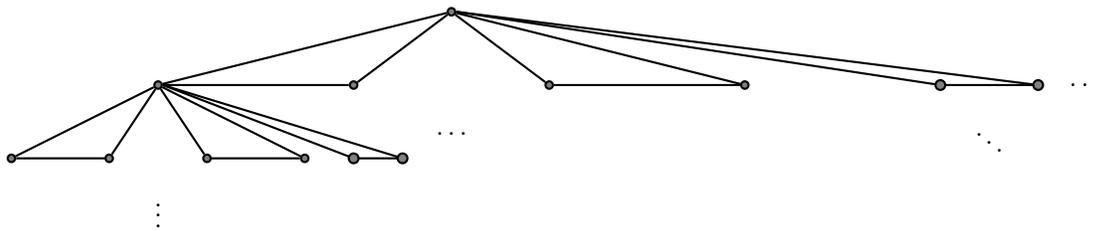
\end{center}
\vspace{-11mm}
\begin{proof}
Let $\kappa$ be an infinite cardinal. We will construct the graph $\Gamma^\kappa$ as follows: 
\vspace{.5mm}

\noindent Step 0: Start with one vertex.

\noindent Step n: To every vertex of finite degree, adjoin $\kappa$ many triangles. 

\vspace{2mm}

The resulting graph after $\omega$ many steps is $\Gamma^\kappa$. (see \cref{Fig1}) It is clear that $\Gamma$ is distance-transitive (and thus also vertex-transitive). To see this, let $n\in \N$ and $u,v,w$ be such that $d(u,v) = d(u,w) = n$. The shortest path connecting $u$ and $v$ can be sent to the shortest path connecting $u$ and $w$. Each edge in this path is part of a unique triangle since otherwise, the path would not be a shortest path. Sending each triangle making up the path from $u$ to $v$ to the corresponding triangle making up the path from $u$ to $w$ would induce a bijection.

Also, note that the graph's valency and diameter are both greater than 2. Thus, (infinite) Smith's Theorem holds. Since its diameter is infinite, $\Gamma^\kappa$ is not antipodal. Since it contains (many) odd cycles, $\Gamma^\kappa$ us not bipartite. Thus $\Gamma^\kappa$ is primitive. 

Since each vertex has degree $\kappa$ and the graph has infinite diameter, given a vertex $v\in V(\Gamma)$ we have $|\Gamma^\kappa_n(v)| = |\kappa \cdot \dots \cdot \kappa| = \kappa$. Thus $\Gamma^\kappa$ is a primitive vertex-transitive graph with distance sequence $1,\kappa,\kappa,\dots$. 
\end{proof}

Next we will give a construction for a primitive vertex-transitive graph with distance sequence $1,\kappa,\dots,\kappa$. In order to do this, we will first re-prove some lemmas of Smith from \cite{Smith}.  

\begin{lemma}[Smith]\label{lem:blocks}
Let $\Gamma$ be a distance-transitive graph with block $B$ and $u \in B$. If $B$ contains a vertex of $\Gamma_i(u)$, then $\Gamma_i(u)\subseteq B$.
\end{lemma}

\begin{proof}
Let $v\in B\cap \Gamma_i(u)$. Let $w \in \Gamma_i(u)$. Since $\Gamma$ is distance-transitive, there is an automorphism $f$ fixing $u$ and sending $v$ to $w$. Since the block $B$ must be invariant under $f$, we have $w \in B$. Thus $\Gamma_i(u)\subseteq B$. 
\end{proof}

\begin{lemma}[Smith]\label{blocksImplyWhole}
Suppose $\Gamma$ is  distance-transitive, $B$ a block of $\Gamma$, and $u,v\in B$ with $d(u,v)=1$. Then $B = V(\Gamma)$. 
\end{lemma}

\begin{proof}
By \cref{lem:blocks}, $\Gamma_1(u)\subseteq B$ and $\Gamma_1(v)\subseteq B$. If $\Gamma$ is the complete graph we are done, so assume not. Then $\Gamma_2(u)$ must be nonempty. Also, $\Gamma_2(u)\cap \Gamma_1(v)$ must be nonempty as well, by distance transitivity.   Let $w \in \Gamma_2(u)\cap \Gamma_1(v)$. Since $w \in \Gamma_1(v)$, $ w \in B$. By \cref{lem:blocks}, thus $\Gamma_2(u)\subseteq B$. Repeating this process gives us $\Gamma_n(v)\subseteq B$ for all $n$. Since $\Gamma$ is connected, thus $B = V(\Gamma)$. 
\end{proof}

Now we will prove the following theorem. 

\begin{theorem}\label{contantKappa}
There is a primitive graph with distance sequence $1,\kappa,\dots,\kappa$ ($\kappa$ appears $n$ times) for any infinite cardinal $\kappa$ and $n\in \N$. 
\end{theorem}

\begin{proof}
Let $\kappa$ be an infinite cardinal and $n\in \N$.

Let $X$ be a set such that $|X| = \kappa$. Let $[X]^n := \{S\subset X : |S| = n\}$ denote the set of subsets of $X$ with $n$ elements.  Define an edge relation $E$ as follows: 

$$u E v \iff (u\setminus v)\cup (v\setminus u) = 2 $$ 

Let $\Gamma^{\kappa,n}:= ([X]^n,E)$. We want to show that $\Gamma^{\kappa,n}$ is distance-transitive. If $n = 1$ then $\Gamma^{\kappa,n}$ is the complete graph and this is trivial. Thus, let $n \geq 2$. 

Let $c,d\in V(\Gamma)^{\kappa,n}$ be such that $d(c,d) = m$. Let $\varphi$ be a bijection on $X$ sending the elements of $a\cap b$ to the elements of $c\cap d$. This is possible since $|a\cap b| = |c\cap d| = n-2m$. and sending elements of $a\setminus b$ to $c \setminus d$ and likewise $b \setminus a$ to $d \setminus c$. This bijection induces an automorphism on $\Gamma^{\kappa,n}$ sending $a$ to $b$ and $c$ to $d$. Thus $\Gamma^{\kappa,n}$ is distance-transitive.  

Since distance transitivity implies vertex transitivity, $\Gamma^{\kappa,n}$ is also vertex-transitive. 

Next we want to show that $\Gamma^{\kappa,n}$ is primitive. When $n=1$ this is trivial. Let $n = 2$. Assume $V(\Gamma) = B_0 \cup \dots \cup B_k$ or $V(\Gamma) = B_0 \cup B_2 \cup \cdots \cup B_{\lfloor \frac{n}{2}\rfloor} $ is a nontrivial block system of $\Gamma^{\kappa,n}$. Let $u,v\in B_0$ and let $u = \{u_0,u_1\}$ and $v = \{v_0,v_1\}$. The sets $u$ and $v$ must be disjoint, otherwise by \cref{blocksImplyWhole} the blocks are trivial. Note that $B_0$ cannot contain all vertices disjoint from both $u$ and $v$, since for any $x\in X\setminus (u\cup v)$ we have that $\{x,u_0\}$ is disjoint from $v$, but if $\{x,u_0\}$ were in $B_0$, it would give us a contradiction from \cref{blocksImplyWhole}. Thus there is a $w = \{w_0,w_1\}$ disjoint from $v$ which is not in $B_0$. Let $p$ be the permutation of $X$ swapping $v_0$ and $w_0$, $v_1$ and $w_1$ and fixing everything else. The induced automorphism of $\Gamma^{\kappa,n}$ switches $v$ and $w$ while fixing $u$, contradicting that this is a block system. 

Let $n \geq 3$. Since the diameter and valency are both greater than 2, we can use Smith's theorem.  Let $v = \{v_0,v_1,\dots,v_{n-1}\}$ be a vertex of $\Gamma^{\kappa,n}$. Let $b,c\in X\setminus v$. Consider $w_1 = \{b, v_1,\dots,v_{n-1}\}$, $w_2 = \{c, v_1,\dots,v_{n-1}\}$. Clearly $d(v,w_1) = d(w_1,w_2) = d(w_2,v) = 1$, thus $v,w_1,w_2$ create an odd cycle and so $\Gamma^{\kappa,n}$ is not bipartite. 

Let $v\in V(\Gamma^{\kappa,n})$, $u\in \Gamma^{\kappa,n}_n(v)$ such that $i = \{u_0,\dots, u_{n-1}\}$. Let $a \in X \setminus(v\cup w)$,. Consider $u' = \{a,u_1,\dots,u_{n-1}\}$. Then $w'\in \Gamma^{\kappa,n}_n(v)$ and $d(w,w') = 1$. Thus $\Gamma^{\kappa,n}$ is not antipodal. Then, by Smith's Theorem, $\Gamma^{\kappa,n}$ is primitive. 

All that is left to show is that $\Gamma^{\kappa,n}$ has the desired distance sequence. Let $0 < m \leq n$, $v\in V(\Gamma)$. Then we have that $\Gamma_m(v) = \{w\in V(\Gamma): |v\triangle w|= 2m\}$. Consider $w \mapsto w\setminus v$. This is a bijection from $\Gamma_m(v) \to [X]^m$. Thus $|\Gamma_m(v)|=|[X]^m|=\kappa$, and therefore the $m$'th element in the distance sequence is $d$. Since $m<n$ was arbitrary, the distance sequence for $\Gamma$ is $1,\kappa,\dots ,\kappa$. 
\end{proof}

\section{Countable Graphs}\label{sec: countable}

\noindent The goal of this section will be to prove the following theorem:
 \begin{theorem}\label{countable2}
    Let $\Gamma$ be a locally infinite but countable primitive graph. If $\Gamma$ has infinite diameter, its distance sequence is $1,\aleph_0,\aleph_0,\dots$. If $\Gamma$ has finite diameter then $\Gamma$ has distance sequence $1,\aleph_0,\dots,\aleph_0,a$ for $2\leq a \leq \aleph_0$.  
    \end{theorem}

It was shown in \cite{Pegden} that if $\Gamma$ is a locally infinite vertex-transitive graph of degree $\kappa$ with infinite diameter then it has distance sequence $1,\kappa,\kappa,\dots$ and if it is of finite degree, then it has distance sequence $1,\kappa,\dots,\kappa,\alpha$ for $\alpha\leq \kappa$. Thus, for countable graphs, it suffices to show that in the second case that $\alpha$, which in this case is a natural number, is greater than 1. 

\begin{proof}
Let $\Gamma$ be such a graph with distance sequence $1,\aleph_0,\dots,\aleph_0,1$ and diameter $n$. Let $v\in V(\Gamma)$. Then there is a unique $v'\in V(\Gamma)$ such that $d(v,v')= n$. For $v'$, $v$ is also the unique element at distance $n$. 

Define a relation $\sim$ where for $v,w\in V(\Gamma)$, $v\sim w\iff d(v,w) = n$  or $ d(v,w) = 0$. The relation $\sim$ is clearly symmetric and reflexive. The relation $\sim$ is also vacuously transitive since there are no distinct $u,v,w\in V(\Gamma)$ such that $u \sim v$ and $v\sim w$. Thus $\sim$ is an equivalence relation and its equivalence classes form a partition. 

The partition $V(\Gamma)/\sim$ is preserved by any automorphism since $\sim$ is defined by the distance and distance must be preserved. Therefore $V(\Gamma)/\sim$ is a block system of $\Gamma$ and thus $\Gamma$ is imprimitive.
\end{proof}

Next, we will prove a few more lemmas that will aid in our constructions. Since the automorphism group of a graph is the same as for its complement, we have \cref{lem:complement}. 

\begin{lemma}\label{lem:complement}
A graph $G$ is primitive if and only if its complement is primitive. 
\end{lemma}

\begin{lemma}\label{lem:extend}
If there exists a primitive graph with distance sequence $1,\aleph_0,a$ for some $a\in \N$, then there exists a primitive graph with distance sequence $1,\aleph_0,\dots,\aleph_0,a $ and diameter $m+1$ for any $m\in\N$.  
\end{lemma}

\begin{proof}
Let $\Gamma$ be a graph with distance sequence $1,\aleph_0,a$. Let $\Gamma'$ be its complement. Let $d'$ denote the distance function of $\Gamma'$.
Define an edge relation $E^m$ on $V(\Gamma)$ as follows: 

$$v E^m w \iff d'(v,w)>1 \textrm{ and } d'(v,w)  \equiv 1 \pmod{m}$$

Consider the graph $\Gamma^{m} = (V(\Gamma),E^{m})$. Let $d^{m}$ denote distance in this graph.  It is clear that $\Gamma^{m}$ is primitive, since
$ \textrm{Aut} (\Gamma) = 
\textrm{Aut}  (\Gamma')\subseteq  \textrm{Aut}  (\Gamma^m)$. All that is left to show is that $\Gamma^m$ has distance sequence $1,\aleph_0, \dots, \aleph_0,a $.

Given any $v\in V(\Gamma)$, there are infinitely many vertices in  $\bigcup _{k\in \N\setminus\{0\}}\Gamma'_{mk+1}(v)$, so there are infinitely many vertices at $d^m$-distance 1 from $v$. There are infinitely many vertices at $d'$-distance 1 mod $n$ from any vertex in $\bigcup _{k\in \N\setminus\{0\}}\Gamma'_{mk+1}(v)$, so there are again infinitely many vertices at $d^m$-distance 2 from $v$. The same reasoning applies to any entry of the distance sequence until the last entry. Notice any vertex $u$ adjacent to $v$ in $\Gamma'$ are at $d^m$-distance $m+1$ from $v$ since to get from $u$ to $v$ in $\Gamma^m$, the path would have to use three edges to get from $u$ to a point at $d'$-distance $m^2+1$ from $v$, which would then be at $d^m$-distance 1 from $v$. This is clearly the shortest path and thus $d'(u,v) = m+1$. Since the vertices adjacent to $v$ in $\Gamma'$ are the only vertices at  $d^m$-distance $m+1$ from $v$, the distance sequence of $\Gamma^m$ would be $1,\aleph_0,\dots, \aleph_0, a$.  
\end{proof}

Next we will give a construction for primitive graphs with distance sequence $1,\aleph_0,\dots,\aleph_0,a$ for $a$ composite and diameter $n$. 

\begin{theorem}\label{composite}
There exists a vertex-transitive, primitive graph with distance sequence $1,\aleph_0,n$ for composite $n >4$. 
\end{theorem}

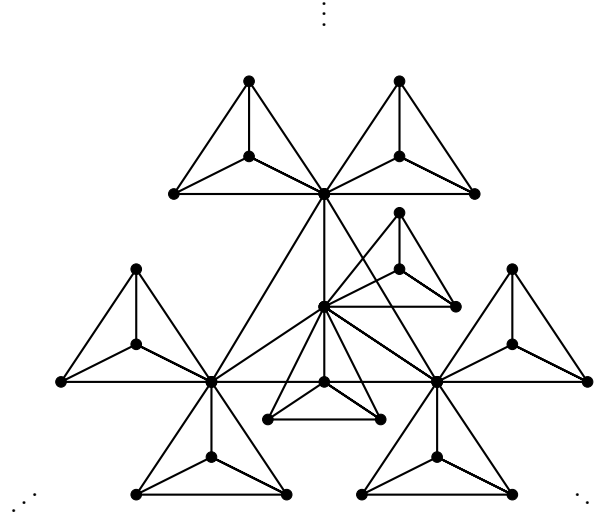
\begin{figure}
\begin{center}
\begin{tikzpicture}[scale=1]
\node[fill = white, draw = white] at (2,2.5) {$ \vdots$};

\node[fill = white, draw = white] at (-2,-4) {$\iddots$};

\node[fill = white, draw = white] at (5.5,-4) {$\ddots$};

\draw[fill=black] (0,0) circle (2pt);
\draw[fill=black] (2,0) circle (2pt);
\draw[fill=black] (1,.5) circle (2pt);
\draw[fill=black] (1,1.5) circle (2pt);

\draw[fill=black] (2,0) circle (2pt);
\draw[fill=black] (4,0) circle (2pt);
\draw[fill=black] (3,.5) circle (2pt);
\draw[fill=black] (3,1.5) circle (2pt);

\draw[fill=black] (.5,-2.5) circle (2pt);
\draw[fill=black] (3.5,-2.5) circle (2pt);
\draw[fill=black] (2,-1.5) circle (2pt);
\draw[fill=black] (2,0) circle (2pt);

\draw[fill=black] (-1.5,-2.5) circle (2pt);
\draw[fill=black] (.5,-2.5) circle (2pt);
\draw[fill=black] (-.5,-2) circle (2pt);
\draw[fill=black] (-.5,-1) circle (2pt);

\draw[fill=black] (-.5,-4) circle (2pt);
\draw[fill=black] (1.5,-4) circle (2pt);
\draw[fill=black] (.5,-3.5) circle (2pt);
\draw[fill=black] (.5,-2.5) circle (2pt);

\draw[fill=black] (3.5,-2.5) circle (2pt);
\draw[fill=black] (5.5,-2.5) circle (2pt);
\draw[fill=black] (4.5,-2) circle (2pt);
\draw[fill=black] (4.5,-1) circle (2pt);

\draw[fill=black] (2.5,-4) circle (2pt);
\draw[fill=black] (4.5,-4) circle (2pt);
\draw[fill=black] (3.5,-3.5) circle (2pt);
\draw[fill=black] (3.5,-2.5) circle (2pt);

\draw[fill=black] (2,-1.5) circle (2pt);
\draw[fill=black] (3.75,-1.5) circle (2pt);
\draw[fill=black] (3,-1) circle (2pt);
\draw[fill=black] (3,-.25) circle (2pt);

\draw[fill=black] (1.25,-3) circle (2pt);
\draw[fill=black] (2.75,-3) circle (2pt);
\draw[fill=black] (2,-2.5) circle (2pt);
\draw[fill=black] (2,-1.5) circle (2pt);

\draw[thick] (0,0) -- (2,0) -- (1,.5) -- (0,0) -- (1,1.5) -- (2,0) -- (1,.5) -- (1,1.5);
\draw[thick] (2,0) -- (4,0) -- (3,.5) -- (2,0) -- (3,1.5) -- (4,0) -- (3,.5) -- (3,1.5);
\draw[thick] (.5,-2.5) -- (3.5,-2.5) -- (2,-1.5) -- (.5,-2.5) -- (2,0) -- (3.5,-2.5) -- (2,-1.5) -- (2,0);
\draw[thick] (-1.5,-2.5) -- (.5,-2.5) -- (-.5,-2) -- (-1.5,-2.5) -- (-.5,-1) -- (.5,-2.5) -- (-.5,-2) -- (-.5,-1);
\draw[thick] (-.5,-4) -- (1.5,-4) -- (.5,-3.5) -- (-.5,-4) -- (.5,-2.5) -- (1.5,-4) -- (.5,-3.5) -- (.5,-2.5);
\draw[thick] (3.5,-2.5) -- (5.5,-2.5) -- (4.5,-2) -- (3.5,-2.5) -- (4.5,-1) -- (5.5,-2.5) -- (4.5,-2) -- (4.5,-1);
\draw[thick] (2.5,-4) -- (4.5,-4) -- (3.5,-3.5) -- (2.5,-4) -- (3.5,-2.5) -- (4.5,-4) -- (3.5,-3.5) -- (3.5,-2.5);
\draw[thick] (2,-1.5) -- (3.75,-1.5) -- (3,-1) -- (2,-1.5) -- (3,-.25) -- (3.75,-1.5) -- (3,-1) -- (3,-.25);
\draw[thick] (1.25,-3) -- (2.75,-3) -- (2,-2.5) -- (1.25,-3) -- (2,-1.5) -- (2.75,-3) -- (2,-2.5) -- (2,-1.5);
\end{tikzpicture}
\end{center}
\caption{A finite but enlightening subgraph of the 3-tree of 4-cliques}
\end{figure}

\begin{proof}
Let $n = a(b-1)$ for $a,b\in \N$ such that $a,b > 2$. Consider the $a$-tree of $b$-cliques. This is the tree one would get by starting with a $b$-clique and at step 1 and adding $(a-1)$-many more $b$-cliques at each vertex of degree $b-1$. At the next step, start with this resulting graph and again add $(a-1)$-many more $b$-cliques at each vertex of degree $b-1$. The resulting graph after infinitely many steps is the $a$-tree of $b$-cliques, denoted $T(a,b)$.

Note that the degree of $T(a,b)$ is $ab$. It is known from \cite{Babai} that such a tree is distance-transitive. Since it contains odd cycles it is not bipartite, and since it has infinite diameter it is not antipodal. Thus, by Smith's Theorem, it is primitive. 

Then, its compliment is a vertex primitive graph with distance sequence $1,\aleph_0, a(b-1)$. 
\end{proof}

\cref{composite} results in the following corollary. 
\begin{cor}
There exists a primitive graph with distance sequence $1,\aleph_0,\dots,\aleph_0,a$ of diameter $n$ for any  composite $a>4$ and $n\in \N$
\end{cor}
This corollary comes from combining \cref{lem:extend} with \cref{composite}.

\section{Further Research}

   It would be of interest to prove the following conjectures. Doing so would complete the classification of distance sequences for primitive graphs. 
   
    \begin{conj}
    There exists a primitive graph with distance sequence $1,\aleph_0,4$.
    \end{conj}
    
    \begin{conj}
    There does not exist a primitive graph with distance sequence $1,\aleph_0,p$ for primes $p$. 
    \end{conj}

\section{Acknowledgements}
 \noindent I would like to thank Wes Pegden for invaluable help and edits.

\end{document}